\documentclass[12pt,reqno]{amsart}

\usepackage[utf8]{inputenc}

\usepackage{amsaddr}
\usepackage{a4wide}
\usepackage{hyperref}

\usepackage{amsmath,amsthm,latexsym,xspace,amscd,amssymb,xy}               
\usepackage{color}
\usepackage{enumerate}

\newtheorem{theorem}{Theorem}[section]
\newtheorem{lemma}[theorem]{Lemma}
\newtheorem{corollary}[theorem]{Corollary}
\newtheorem{proposition}[theorem]{Proposition}

\theoremstyle{definition}
\newtheorem{definition}[theorem]{Definition}
\newtheorem{example}[theorem]{Example}
\newtheorem{remark}[theorem]{Remark}



\author{Daniel Lännström}
\address{Department of Mathematics and Natural Sciences,
Blekinge Institute of Technology,
SE-37179 Karlskrona, Sweden}

\email{{\scriptsize daniel.lannstrom@bth.se}}

\date{\today}

\title[Chain conditions for epsilon-strongly graded rings]{Chain conditions for epsilon-strongly graded rings with applications to leavitt path algebras}

\begin{document}


\begin{abstract}
Let $G$ be a group with neutral element $e$ and let $S=\bigoplus_{g \in G}S_g$ be a $G$-graded ring. A necessary condition for $S$ to be noetherian is that the principal component $S_e$ is noetherian. The following partial converse is well-known: If $S$ is strongly-graded and $G$ is a polycyclic-by-finite group, then $S_e$ being noetherian implies that $S$ is noetherian. We will generalize the noetherianity result to the recently introduced class of epsilon-strongly graded rings. We will also provide results on the artinianity of epsilon-strongly graded rings.

As our main application we obtain characterizations of noetherian and artinian Leavitt path algebras with coefficients in a general unital ring. This extends a recent characterization by Steinberg for Leavitt path algebras with coefficients in a commutative unital ring and previous characterizations by Abrams, Aranda Pino and Siles Molina for Leavitt path algebras with coefficients in a field. Secondly, we obtain characterizations of noetherian and artinian unital partial crossed products. 

\end{abstract}

\maketitle

\pagestyle{headings}

\section{Introduction}
\label{sec:1}

Let $R$ be an associative ring equipped with a multiplicative identity $1 \ne 0$ and let $G$ be a group with neutral element $e$. It is straightforward to show that if the group ring $R[G]$ is right (left) noetherian, then $R$ is right (left)  noetherian and $G$ is a right (left) noetherian group, i.e. satisfies the ascending chain condition on right (left) subgroups. In 1954, Hall \cite{hall1954finiteness} proved that if $G$ is polycyclic-by-finite and $R$ is right (left) noetherian, then $R[G]$ is right (left) noetherian. Motivated by this partial result, Bovdi \cite[Sect. 1.148]{filippov1993dniester} asked for a complete converse: If $R$ and $G$ are right (left) noetherian, is $R[G]$ right (left) noetherian? This question was settled in 1989 when Ivanov \cite{ivanov1989group} gave an example of a right noetherian group $G$ such that the group ring $R[G]$ is not right noetherian for any right noetherian base ring $R$, proving that the converse does not hold for general noetherian groups. On the other hand, it is not known if polycylic-by-finite groups is the largest class for which the converse holds. Determining for which groups $G$, noetherianity of $R$ is a sufficient condition for the group ring $R[G]$ to be noetherian is still an open question (see \cite[Sect. 2]{brown2014hopf}). For artinianity, a more definitive answer was obtained by Connell \cite{connell1963group} in 1963. Namely, $R[G]$ is left (right) artinian if and only if $R$ is left (right) artinian and $G$ is finite.

Let $\alpha \colon G \to \text{Aut}(R)$ be a group homomorphism. Recall that the \textit{skew group ring} $R \star_\alpha G$ has the same additive structure as $R[G]$ but the multiplication is skewed by the group action. More precisely, the multiplication of monomials is defined by $(a e_g)(b e_h) = a \alpha_g(b) e_{gh}$. It was proved by Park \cite{park1979artinian} that $R \star_\alpha G$ is left (right) artinian if and only if $R$ is left (right) artinian and $G$ is finite. Skew group rings are examples of so-called algebraic crossed products. Unfortunately, the characterization by Park does not generalize to general crossed products. Passman \cite{passman1970radicals} has given examples of artinian twisted group rings by an infinite $p$-group showing that $G$ being finite is not necessary for a crossed product to be artinian. 

A generalization of crossed products is the notion of a strongly graded ring. Recall that a ring $S$ is \emph{graded} by the group $G$ if $S=\bigoplus_{g \in G} S_g$ for additive subsets $S_g$ of $S$ and $S_g S_h \subseteq S_{gh}$ for all $g,h \in G$. If $S_g S_h = S_{gh}$ for all $g,h \in G$, then $S$ is \textit{strongly graded} by $G$. There is a rigid relation between the \textit{principal component} $R = S_e$ and the whole ring $S$. In particular, it turns out that the characterization of noetherianity by Hall \cite{hall1954finiteness}, generalizes to strongly graded rings. Namely, let $S=\bigoplus_{g \in G} S_g$ be a strongly graded ring where $G$ is a polycyclic-by-finite group and let $R=S_e$ be the principal component. Then $S$ is left (right) noetherian if and only if $R$ is left (right) noetherian. This can be proved using Dade's Theorem (see \cite[Thm. 5.4.8]{nastasescu2004methods}) or with a Hilbert Basis Theorem argument (see \cite[Prop. 2.5]{bell1987localization}). For artinian strongly graded rings, even though Passman's example makes a full generalization impossible, Saorín \cite{saorin1992descending} has obtained some results similar to those of Connell and Park, but with heavy conditions on the principal component. 

\smallskip

Loosely speaking, all of these classical objects: group rings, skew group rings and crossed products are derived from some kind of action of a group $G$ acting on a ring $R$. A recent development has been to consider the partial analogues of these objects, coming from a partial action of $G$ on $R$. We will elaborate more on these objects later in the introduction. 

\begin{figure}[ht]
\begin{tabular}{ c | c }
Classical objects & Generalizations \\
\hline
actions &  partial actions \\
skew group rings & partial skew group rings \\
crossed products & unital partial crossed products  \\
strongly graded rings & epsilon-strongly graded rings 
\end{tabular}
\caption{Classical objects and their partial analogues considered in this paper.}
\label{fig:1}
\end{figure}
The class of \emph{epsilon-strongly} graded rings was introduced by Nystedt, Öinert and Pinedo in \cite{nystedt2016epsilon}. For our purposes, this class is the correct partial analogue of strongly graded rings. We aim to investigate noetherianity and artinianity on the right-hand side in Figure \ref{fig:1}. In this article, we will prove the following characterization of noetherianity.

\begin{theorem}
Let $G$ be a polycyclic-by-finite group and let $S$ be an epsilon-strongly $G$-graded ring. Then $S$ is left (right) noetherian if and only if $R=S_e$ is left (right) noetherian.\label{thm:n}
\end{theorem}
For artinianity, we obtain the following result:
\begin{theorem}
Let $G$ be a torsion-free group and let $S=\bigoplus_{g \in G} S_g$ be an epsilon-strongly $G$-graded ring. Then $S$ is left (right) artinian if and only if $R=S_e$ is left (right) artinian and $S_g = \{ 0 \}$ for all but finitely many $g \in G$. 
\label{thm:a}
\end{theorem}

\smallskip

Partial actions were first introduced by Exel \cite{exel1994circle} in the early 1990s to study $C^*$-algebras. For a survey of the history of partial actions, see \cite{dokuchaev2011partial,batista2016partial}. Later, partial actions on a ring $R$ were considered by Dokuchaev and Exel in \cite{dokuchaev2005associativity} to construct partial skew group rings as analogues of the classical skew group rings. The partial skew group ring of an arbitrary partial action is not necessarily associative (see \cite[Expl. 3.5]{dokuchaev2005associativity}).  Analogously, twisted partial actions were considered in \cite{dokuchaev2008crossed} and shown to give rise to partial crossed products. 

Among the partial crossed products originally considered by Dokuchaev and Exel, the subclass coming from so-called unital twisted partial actions (see Section 5 for a definition) was considerd in \cite{bagio2010crossed} and shown to be especially well-behaved. For instance, these crossed products are always associative and unital algebras (see Section 5). In the sequel, this type of crossed products will simply be called \emph{unital partial crossed products}.

\smallskip

Studying partial skew group rings, Carvalho, Cortes and Ferrero \cite{carvalho2011partial} obtained the following generalization from the classical setting. 
\begin{theorem}(\cite[Cor. 3.4]{carvalho2011partial})
Let $\alpha$ be a unital partial action of a polycyclic-by-finite group $G$ on a ring $R$. If $R$ is right (left) noetherian, then the partial skew group ring $R \star_\alpha G$ is right (left) noetherian.
\label{thm:partial_skew_groups_noeth}
\end{theorem} 

Using the techniques in this paper, we will extend their result to unital partial crossed products (see Corollary \ref{thm:crossed_noeth}).  


\subsection{Leavitt path algebras}

The Leavitt path algebra is an algebra attached to a directed graph $E$. Surprisingly, the Leavitt path algebra of a finite graph is naturally epsilon-strongly $\mathbb{Z}$-graded (see \cite{nystedt2017epsilon}). Leavitt path algebras were first introduced in \cite{abrams2005leavitt} as an algebraic analogue of graph $C^*$-algebras. The monograph \cite{abrams2017leavitt} by Abrams, Ara and Siles Molina is a comprehensive general reference. The main focus of the research has been Leavitt path algebras with coefficients in a field. In \cite{tomforde2009leavitt}, Leavitt path algebras over a unital commutative ring were considered. They have been further studied in \cite{larki2015ideal} and \cite{katsov2017simpleness}. In this paper we will follow \cite{hazrat2013graded}, \cite{nystedt2017epsilon} and consider Leavitt path algebras with coefficients in an arbitrary unital ring.  

In the case of Leavitt path algebras with coefficients in a field, there has been much research connecting properties of the underlying graph with algebraic properties of the corresponding algebra. One particular question is how finiteness conditions on the algebra relates to finiteness conditions on the graph. In \cite{abrams2008locally} it is shown that the Leavitt path algebra of a finite graph $E$ is noetherian if and only if the graph $E$ does not have any cycles with exits. Furthermore, it was proven in \cite{ABRAMS2007753} that the Leavitt path algebra is artinian if and only if the graph is finite acyclic. Another proof of this result is given in \cite{NYSTEDT2018433}. 

Noetherianity and artianity of Leavitt path algebras with coefficients in a commutative ring has been characterized by Steinberg \cite{steinberg2018chain}. Using a different technique, we will extend his characterization to Leavitt path algebras with cofficients in a general unital ring. This also generalizes previous results by Abrams, Aranda Pino and Siles Molina, cf. \cite{abrams2008locally,abrams2010chain,ABRAMS2007753,abrams2017leavitt}.

\begin{theorem}
Let $E$ be a directed graph and let $R$ be a unital ring. Consider the Leavitt path algebra $L_R(E)$ with coefficients in $R$. The following assertions hold:
\begin{enumerate}[(a)]
\begin{item}
$L_R(E)$ is a left (right) noetherian unital ring if and only if $E$ is finite and satisfies Condition (NE) and $R$ is left (right) noetherian.
\end{item}
\begin{item}
$L_R(E)$ is a left (right) artinian unital ring if and only if $E$ is finite acyclic and $R$ is left (right) artinian.
\end{item}
\begin{item}
If $L_R(E)$ is a semisimple unital ring, then $E$ is finite acyclic and $R$ is semisimple. Conversely, if $R$ is semisimple with $n \cdot 1_R$ invertible for every integer $n \ne 0$ and $E$ is finite acyclic, then $L_R(E)$ is a semisimple unital ring. 
\end{item}
\end{enumerate}
\label{thm:s}
\end{theorem}
\begin{remark}
Steinberg \cite{steinberg2018chain} proves a complete characterization of semisimple Leavitt path algebras with coefficients in a commutative ring. Our assumption in Theorem \ref{thm:s}(c) that $n \cdot 1_R$ is invertible for every $n \ne 0$ is not a necessary condition. 
\end{remark}

Most of these previous studies do not consider the additional structure on the Leavitt path algebra coming from the grading. Here we will follow the approach first taken by Hazrat \cite{hazrat2013graded} and focus on the graded structure. The key point realized by Nystedt and Öinert \cite{nystedt2017epsilon} is that the Leavitt path algebra of a finite graph has a canonical epsilon-strong $\mathbb{Z}$-grading. Since $\mathbb{Z}$ is both torsion-free and polycyclic-by-finite, we can apply our general theorems for epsilon-strongly graded rings.

\smallskip

In Section \ref{sec:2}, we give miscellaneous preliminaries that are needed later on.


In Section \ref{sec:3}, we prove our main results: Theorem \ref{thm:n} and Theorem \ref{thm:a}. In the first part of the section, Theorem \ref{thm:n} is proved by considering the special cases: finite $G$ and infinitely cyclic $G$. Furthermore, Theorem \ref{thm:a} is essentially a consequence of Bergman's famous observation that the Jacobson radical of a $\mathbb{Z}$-graded ring is a graded ideal. 

In Section \ref{sec:4}, we apply the investigation of chain conditions on epsilon-strongly graded rings to Leavitt path algebras with coefficients in a unital ring. Indeed, Corollary \ref{cor:lpa_noeth}, Proposition \ref{prop:lpa_left} and Corollary \ref{cor:full_semi}  generalize previous characterizations of noetherian, artinian and semisimple Leavitt path algebras of a finite graph with coefficients in a field, cf. \cite{abrams2008locally,abrams2010chain,ABRAMS2007753,abrams2017leavitt}. Finally, we prove Theorem \ref{thm:s}, generalizing the characterization given by Steinberg \cite{steinberg2018chain}.

In Section \ref{sec:5}, we apply Theorem \ref{thm:n} and Theorem \ref{thm:a} to unital partial crossed products and obtain characterizations of noetherian and artinian unital partial crossed products: Corollary \ref{thm:crossed_noeth} and Corollary \ref{thm:crossed_artinian}. We will show that Corollary \ref{thm:crossed_noeth} generalizes Theorem \ref{thm:partial_skew_groups_noeth} (cf. \cite[Cor. 3.4]{carvalho2011partial}).

\section{Preliminaries}
\label{sec:2}

In the following we fix a group $G$ with neutral element $e$. Our rings, unless otherwise stated, will be associative and equipped with a multiplicative identity $1 \ne 0$. A ring $S$ is called \emph{$G$-graded} or \emph{graded of type $G$}, if there exists a family $\{ S_g \}_{g \in G}$ of additive subgroups of $S$ such that (i) $S=\bigoplus_{g \in G} S_g$ and (ii) $S_g S_h \subseteq S_{gh}$ for all $g,h \in G$. The additive subgroups $S_g$ are called \textit{homogeneous components} and the elements of $S$ contained in some $S_g$, are called the \textit{homogeneous elements}. The homogeneous component $S_e$ is called the \textit{principal component} of $S$ and will usually be denoted by $R$. The \textit{support} of a $G$-grading $\text{Supp}(S)$ is the set of $g \in G$ such that $S_g \ne \{ 0 \}$. An element in $ x \in S$ decomposes uniquely as a finite sum $x=\sum s_g$ where the $s_g$'s are homogeneous elements. The \emph{support} of an element $x \in S$, denoted by $\text{Supp}(x)$, is the finite set such that $g \in \text{Supp}(x) \iff s_g \ne 0$ in the decomposition of $x$. 

\smallskip

The following definition was introduced \cite[Def. 4.5]{clark2018generalized} in the study of Steinberg algebras.

\begin{definition}
A $G$-graded ring $S=\bigoplus_{g \in G}S_g$ is called \textit{symmetrically graded} if,
\begin{equation}
S_g S_{g^{-1}} S_g  = S_g, \hspace{5em} \forall {g \in G}.
\label{eq:1}
\end{equation}
\label{def:symmetric}
\end{definition}
We point out that strong gradings are symmetric. Interestingly, the Leavitt path algebra of any graph is symmetrically graded (see \cite[Prop. 3.2]{nystedt2017epsilon}). An easy example of a grading that is not symmetric is the standard $\mathbb{Z}$-grading on the polynomial ring $R[X]$, i.e. $S_n = \{ 0 \}$ for $n<0$ and $S_n=R X^n$ for $n \geq 0$. 

\begin{remark}
Let $S$ be a symmetrically graded ring. If $S_g S_{g^{-1}} = \{ 0 \}$ for some $g \in G$, then $ \{0 \} = S_{g^{-1}} S_g S_{g^{-1}} = S_g S_{g^{-1}} S_g,$ hence by (\ref{eq:1}), we get that $S_g = S_{g^{-1}} = \{ 0 \}$. Thus, $g \in \text{Supp}(S) \iff g^{-1} \in \text{Supp}(S).$ 
\end{remark}

\begin{example}
Let $R$ be a ring. An easy check shows that $R[X^2, X^{-2}]= \bigoplus_{n \in \mathbb{Z}} S_n$ where $S_{2k+1}=\{ 0 \}$ and $S_{2k} = R X^{2k}$ is a symmetrical grading. Since $S_{2k+1} S_{-2k-1} = \{ 0 \} \ne R$, the grading is not strong. 
\label{ex:sym}
\end{example}

\smallskip

If $S$ is a $G$-graded ring with principal component $R$, then $S_g S_{g^{-1}}$ is an $R$-ideal for each $g \in G$. The following definition was introduced by Nystedt, Öinert and Pinedo.

\begin{definition}
(\cite[Def. 4 and Prop. 7]{nystedt2016epsilon})
A $G$-graded ring $S=\bigoplus_{g \in G} S_g$ with principal component $R=S_e$ is called \textit{epsilon-strongly graded} if, 
\begin{enumerate}[(a)]
\begin{item}
the grading is symmetrically graded, and,
\end{item}
\begin{item}
$S_g S_{g^{-1}}$ is a unital $R$-ideal for each $g \in G$. 
\end{item}
\end{enumerate}
\label{def:epsilon}
\end{definition}

It is straightforward to see that the ring in Example \ref{ex:sym} is epsilon-strongly graded. We point out that Definition \ref{def:epsilon} is equivalent to the following conditions: (i) $S$ being $G$-graded and (ii) for each $g \in G$, there exists some $\epsilon_g \in S_g S_{g^{-1}}$ such that for all $x \in S_g$,  $ x= \epsilon_g x = x \epsilon_{g^{-1}}$ (see \cite[Prop. 7]{nystedt2016epsilon}).

\bigskip

We can restrict an epsilon-strong grading to a subgroup of the grading group. This process gives us both a new grading and a new ring. The proof of the following proposition is straightforward and left to the reader.

\begin{proposition}
Let $S=\bigoplus_{g \in G} S_g$ be an epsilon-strongly $G$-graded ring with principal component $R=S_e$. Let $H$ be a subgroup of $G$. Then $S(H)=\bigoplus_{g \in H} S_g$ is an epsilon-strongly $H$-graded ring with principal component $S=S_e$.
\label{prop:half}
\end{proposition}

On the other hand, if $N$ is a normal subgroup of $G$, then there is a way to assign a new, induced $G/N$-grading to the same ring $S$. Indeed, let $S=\bigoplus_{g \in G}S_g $ be a $G$-graded ring with principal component $R$. The induced $G/N$-grading is given by $S=\bigoplus_{C \in G/N} S_C$ where $S_C=\bigoplus_{g \in C} S_g$. In particular, the new principal component is $S_{[e]}=\bigoplus_{g \in N} S_g = S(N)$.  One can show that the induced $G/N$-grading is strong if the $G$-grading of $S$ is strong (see \cite[Prop. 1.2.2]{nastasescu2004methods}). 

In the proof of Theorem \ref{thm:n} we will need the following proposition for the reduction step. 

\begin{proposition}(\cite[Prop. 5.4]{induced2018}, cf. \cite[Prop. 2.3]{bell1987localization})
Let $S$ be an epsilon-strongly $G$-graded ring with principal component $R=S_e$. If $R$ is left or right noetherian, then for any normal subgroup $N$ of $G$, the induced $G/N$-grading of $S$ is epsilon-strong.
\label{prop:induced_epsilon}
\end{proposition}

\bigskip

Finally, we recall the following definition. 

\begin{definition}
A group $G$ is \textit{polycyclic-by-finite} if it has a subnormal series, 
\begin{equation*}
1 = G_0 \lhd G_1 \lhd \dots G_{n-1} \lhd G_n \lhd G_{n+1}=G,
\end{equation*}
such that $G/G_n$ is finite and $G_{i+1}/G_i, 1 \leq i \leq n-1$ is a cyclic group. In other words, a polycyclic-by-finite is a group containing a polycyclic subgroup with finite index.
\label{def:1}
\end{definition}

For example, the groups $\mathbb{Z}^r$ for $r > 0$ are polycyclic-by-finite.

\section{Noetherian and artinian epsilon-strongly graded rings}
\label{sec:3}
Recall that a ring is called \emph{left (right) noetherian} if every ascending sequence of left (right) ideals of the ring stabilizes. A left (right) $R$-module $_R M$ ($M_R$) is \emph{left (right) noetherian} if it satisfies the ascending chain condition on left (right) $R$-submodules. A ring $R$ is left (right) noetherian if and only if $R$ is left (right) noetherian as a left (right) module over itself. Another well-known characterization is that a ring $R$ is left (right) noetherian if and only if every finitely generated left (right) $R$-module is left (right) noetherian. Artinian rings and modules have similar characterizations but instead with respect to the descending chain condition.

\subsection{The Ascending Chain Condition}

\smallskip

We first recall the following well-known result (see e.g. \cite[Prop. 5.4.2]{nastasescu2004methods}).


\begin{proposition}
Let $G$ be an arbitrary group and let $S=\bigoplus_{g \in G} S_g$ be a $G$-graded ring with principal component $R$. If $S$ is left (right) noetherian, then $R$ is left (right) noetherian. \label{prop:going_down1}
\end{proposition}
The converse of Proposition \ref{prop:going_down1} does not hold in general as the following example shows.

\begin{example}
(\cite[Expl.  1.1.22]{hazrat2016graded})
Let $S$ be the ring of Laurent polynomials (over some field $K$) in infinitely many variables $X_1, X_2, \dots$. Then, $S$ is strongly graded by extending the standard one variable $\mathbb{Z}$-grading to an infinite variable $\prod_{\mathbb{N}} \mathbb{Z}$-grading. More precisely, an element ${g \in \prod_{\mathbb{N}} \mathbb{Z}}$ can be identified with a sequence where $g(i)$ is the degree in the variable $X_i$. 

Notice that $S$ is not noetherian but $R=S_e=K$ is noetherian since it is a field. 
\end{example}

In the case of epsilon-strongly graded rings there is a converse of Proposition \ref{prop:going_down1} for a certain class of groups. For this purpose, we will generalize the proof for strongly graded rings given in \cite{bell1987localization}. Some changes are needed, but otherwise the proof is similar and uses a version of Hilbert's Basis Theorem. 

\smallskip

To establish the general result we will first consider the special cases of $G$ finite and $G = \mathbb{Z}$. Then, we reduce the general case to these two special cases. 

\smallskip

We recall that if $M$ is a finitely generated projective left $R$-module, then $\text{Hom}_R(M,R)$ is a finitely generated projective right $R$-module. 

\begin{proposition}
Let $G$ be an arbitrary group and let $S=\bigoplus_{g \in G} S_g$ be an epsilon-strongly $G$-graded ring. If ${\text{Supp}(S) < \infty}$, then $S$ is finitely generated as both a left and right $R$-module.
\label{prop:2}
\end{proposition}
\begin{proof}
By \cite[Prop. 7]{nystedt2016epsilon}, for each $g \in G$, $S_g$ is a finitely generated projective left $R$-module isomorphic to $\text{Hom}_R(S_{g^{-1}}, R)$ as a right $R$-module. In particular, $S_g$ is finitely generated as both a left and right $R$-module. Hence, $S=\bigoplus_{g \in G} S_g$ is finitely generated since the support is finite. 
\end{proof}

\begin{proposition}
Let $G$ be a finite group and let $S=\bigoplus_{g \in G}S_g$ be an epsilon-strongly $G$-graded ring with principal component $R$. If $R$ is left (right) noetherian, then $S$ is left (right) noetherian. 
\label{prop:finite}
\end{proposition}
\begin{proof}
Since the support of $S=\bigoplus_{g \in G} S_g$ is finite, Proposition \ref{prop:2} implies that $_RS$ is finitely generated. Therefore, $_RS$ is left (right) noetherian and, in particular, $_SS$ is left (right) noetherian.
\end{proof}
 
Next, we shall prove the special case of $G = \mathbb{Z}$. Indeed, let $S=\bigoplus_{g \in G} S_g$ be an epsilon-strongly $\mathbb{Z}$-graded ring with principal component $R$. We can define a \emph{length} for each element $s \in S$ in the following way. Any element $s \in S$ decomposes as a finite sum $s = \sum_{i \in \text{Supp}(s)} s_i$ where $0 \ne s_i \in S_i$. We note that the set $\text{Supp}(s) \subset \mathbb{Z}$ is finite, hence we can let $a$ and $b$ be the least and greatest integers of $\text{Supp}(s)$ respectively. Then,  
$s = \sum_{i=a}^b s_i'$  where $s_i' = s_i$ if $i \in \text{Supp}(s)$ and $s_i' = 0$ otherwise. We note that $s_a' \ne 0$ and $s_b' \ne 0$. The strictly positive integer $b-a+1$ is called the \emph{length} of $s$.

Now, let $I$ be a right ideal of $S$. For $n \ge 1$, let,
\begin{equation*}
\text{Id}_n(I) = \Big \{ r \in R \mid r+ \sum_{i=-n+1}^{-1} s_i \in I, s_i \in S_i \Big \},
\end{equation*}
where the sum from $i=-n+1$ to $-1$ is to be understood as empty for $n=1$. Informally, $\text{Id}_n(I)$ is the leading coefficients of elements of $I$ with length at most $n$.

\begin{lemma}
Let $S$ be an epsilon-strongly $\mathbb{Z}$-graded ring.
With the notation as above, the following assertions hold.

\begin{enumerate}[(a)]
\begin{item}
For $n \geq 1$, $\text{Id}_n(I)$ is a right ideal of $R$ and $\text{Id}_{n}(I) \subseteq \text{Id}_{m}(I)$ if $1 \le n \le m$.
\end{item}
\begin{item}
Let $J \subseteq I$ be right ideals of $S$ such that $\text{Id}_n(I)=\text{Id}_n(J)$ for all $n \geq 1$. Then, $J=I$. 
\end{item}
\end{enumerate}

\label{lem:2}
\end{lemma}
\begin{proof}
(a): Straightforward.

\smallskip

(b): Seeking a contradiction, suppose that $I \ne J$. Then, we can choose, 
\begin{equation*}
x = \sum_{i=a}^b s_i \in I \setminus J, 
\end{equation*}
with $s_i \in S_i$ such that the length $b-a+1$ is minimal. For the moment let us assume that we can require $b=0$ without loss of generality. Write $x = s_0 + \sum_{i=c}^{-1} s_i$ and note that $x$ has length $1-c$. Then $s_0 \in \text{Id}_{1-c}(I)=\text{Id}_{1-c}(J)$, which implies that there exists some $y = s_0 +\sum_{i=c}^{-1} \tilde{s_i} \in J$. Note that $x-y = \sum_{i=c}^{-1} (s_i - \tilde {s_i} )$ has length $-c$. Since $s_i - \tilde{s_i} \in S_i$ and by the choice of $x$ as the element of minimal length in $I \setminus J$, we have $x-y \in J$. Hence, $x \in J$, which is a contradiction.  

\smallskip

To finish the proof we need to show that we can assume $b=0$. Again, let $x=\sum_{i=a}^b s_i$ be a fixed element in $I \setminus J$ of minimal length. Using that $I$ is a right ideal of $S$, we get that $x S_{-b} \subseteq I.$  For the moment assume that $xS_{-b} \not \subseteq J$. This implies that there exists some $y_{-b} \in S_{-b}$ such that $x y_{-b} \in I \setminus J$. Therefore, $ x y_{-b} =  \sum_{i=a}^b s_i y_{-b}$ with $s_i y_{-b} \in S_{i}S_{-b} \subseteq S_{i-b}$. Hence, letting $s_i' = s_{i+b}y_{-b} \in S_i$ we see that $x y_{-b} = \sum_{i=a-b}^0 s_i' = s_0' + \sum_{i=a-b}^{-1} s_i'$ has length $b-a+1$. In other words, we can choose $x$ with $b=0$ in the decomposition.  

Finally, we prove that $x S_{-b} \not \subseteq J$ for any $x \in I \setminus J$ of minimal length. Suppose to get a contradiction that $xS_{-b} \subseteq J$. Then $x S_{-b} S_b \subseteq J$. Furthermore, since, $\epsilon_{-b} \in S_{-b}S_b$, 
\begin{equation*}
x \epsilon_{-b} = \sum_{i=a}^b s_i \epsilon_{-b} = s_b + \sum_{i=a}^{b-1} s_i \epsilon_{-b} \in J\subseteq I.
\end{equation*}
Thus, 
\begin{equation*}
x - x \epsilon_{-b} = \Big ( s_b + \sum_{i=a}^{b-1} s_i \Big ) - \Big ( s_b + \sum_{i=a}^{b-1} s_i \epsilon_{-b} \Big )= \sum_{i=a}^{b-1} s_i(1-\epsilon_{-b}) \in I.
\end{equation*}
Since $(1-\epsilon_{-b} ) \in R$ we have that $s_i (1-\epsilon_{-b} ) \in S_i$. But since the length of $x -x \epsilon_{-b}$ is $b-a$, which is strictly less than the the length of $x$, the choice of $x$ implies that $x - x \epsilon_{-b} \in J$. Hence $x \in J$, which is a contradiction.

\end{proof}


\begin{proposition}
Let $S=\bigoplus_{i \in \mathbb{Z}} S_i$ be an epsilon-strongly $\mathbb{Z}$-graded ring and let $R$ be the principal component of $S$. If $R$ is right (left) noetherian, then $S$ is right (left) noetherian. 
\label{prop:case2}
\end{proposition}
\begin{proof}
We will only show the right-handed case as the left-handed case can be proved analogously. 

Assume that $I_1 \subseteq I_2 \subseteq I_3 \subseteq \dots $ is an ascending sequence of right ideals in $S$. Then the diagonal sequence,
\begin{equation*}
\text{Id}_1(I_1) \subseteq \text{Id}_2(I_2) \subseteq \text{Id}_3(I_3) \subseteq \dots,
\end{equation*} is an ascending sequence of right ideals in $R$. Since $R$ is right noetherian by assumption there exists some $n$ such that $\text{Id}_n(I_n) = \text{Id}_k(I_k)$ for any $k \ge n$. 

Moreover, for $1 \le i \le n-1$, consider the sequence,
\begin{equation*}
\text{Id}_i(I_1) \subseteq \text{Id}_i(I_2) \subseteq \text{Id}_i(I_3) \subseteq \dots.
\end{equation*} For each $i$ there exists $n_i$ such that $\text{Id}_i(I_{n_i})=\text{Id}_i(I_m)$ for $m \ge n_i$. Taking, \begin{equation*}
n' = \text{max}_{1\le i \le n -1}(n, n_i),
\end{equation*} we have that $\text{Id}_i(I_{n'})=\text{Id}_i(I_m)$ for $m \ge n'$ and all $i$. By Lemma \ref{lem:2}, $I_m=I_{n'}$. Hence, the original sequence stabilizes and $S$ is right noetherian.   
\end{proof}

We now consider the general case of polycyclic-by-finite groups.

\begin{theorem}
Let $G$ be a polycyclic-by-finite group and let $S=\bigoplus_{g \in G} S_g$ be an epsilon-strongly $G$-graded ring with principal component $R$. If $R$ is right (left) noetherian, then $S$ is right (left) noetherian. 
\label{thm:noeth}
\end{theorem}
\begin{proof}
We only prove the right-handed case since the left-handed case can be treated analogously. Suppose $R$ is right noetherian. We reduce the general proof to the cases where $G=\mathbb{Z}$ or $G$ is finite. The proof goes by induction on the length $n$ of the subnormal series. 

For $n=1$, we have $1 = G_0 \lhd G_1 = G$, i.e. $G=G/1$ is finite. Thus, $S$ is right noetherian by Proposition \ref{prop:finite}.

Next, assume that the theorem holds for subnormal series of length $k$ for some $k > 0$. Let $G$ be a polycyclic-by-finite group with subnormal series, 
\begin{equation*}
1 = G_0 \lhd G_1 \lhd \dots \lhd G_k \lhd G_{k+1}=G, 
\end{equation*}
as in Definition \ref{def:1}. By Proposition \ref{prop:half}, $S(G_k)$ is epsilon-strongly $G_k$-graded with principal component $R$. By the induction hypothesis, $S(G_k)$ is right noetherian. Furthermore, $S(G_{k+1})$ is epsilon-strongly $G_{k+1}$-graded with principal component $R$. Since $R$ is right noetherian, Proposition \ref{prop:induced_epsilon} implies that, the induced $G_{k+1}/G_{k}$-grading is epsilon-strong with principal component $S(G_k)$. By the assumption on $G$, $G_{k+1}/G_k$ is either (i) isomorphic to $\mathbb{Z}$ or (ii) finite. In the first case, Proposition \ref{prop:case2} implies that $S(G_{k+1})$ is right noetherian. In the second case, Proposition \ref{prop:finite} implies that $S(G_{k+1})$ is right noetherian.   

Hence, $S=S(G)=S(G_{k+1})$ is right noetherian and the theorem follows by the induction principle.
\end{proof}

We can now establish our characterization of noetherian epsilon-strongly graded rings.

\begin{proof}[Proof of Theorem \ref{thm:n}]
Assume that $S$ is left (right) noetherian. Then, Proposition \ref{prop:going_down1} implies that $R=S_e$ is left (right) noetherian. Conversely, let the principal component $R=S_e$ be left (right) noetherian. Then, Theorem \ref{thm:noeth} implies that $S$ is left (right) noetherian.
\end{proof}

\subsection{The Descending Chain Condition}

Given a torsion-free group $G$ we characterize when an epsilon-strongly $G$-graded ring is one-sided artinian.

\smallskip

The following well-known result gives a necessary condition for a group graded ring to be one-sided artinian (see e.g. \cite[Prop. 5.4.2]{nastasescu2004methods}).

\begin{proposition}
Let $G$ be an arbitrary group and let $S=\bigoplus_{g \in G} S_g$ be a $G$-graded ring with principal component $R$. If $S$ is left (right) artinian, then $R$ is left (right) artinian. 
\label{prop:artinian}
\end{proposition}

First, we prove the sufficiency of the conditions in Theorem \ref{thm:a}.

\begin{proposition}
Let $G$ be an arbitrary group and let $S=\bigoplus_{g \in G} S_g$ be an epsilon-strongly $G$-graded ring. If $\text{Supp}(S) < \infty$ and $R$ is left (right) artinian, then $S$ is left (right) artinian.\label{prop:suff}
\end{proposition}
\begin{proof}
By Proposition \ref{prop:2}, $_RS$ is finitely generated and thus $_RS$ is left artinian. In particular, $_SS$ is left artinian.
\end{proof}
Considering the case where $G$ is torsion-free we can use the following general theorem (see \cite[Thm. 9.6.1]{nastasescu2004methods}) to get a converse to Proposition \ref{prop:suff}.
\begin{theorem}
Let $G$ be a torsion-free group and let $S=\bigoplus_{g \in G} S_g$ be a $G$-graded group. If $S$ is a left (right) artinian, then $\text{Supp}(S) < \infty$. 
\label{thm:big}
\end{theorem} 
\begin{remark}
We show that the right case of Theorem \ref{thm:big} follows from the left case. Letting $S$ be a right artinian ring, the opposite ring $S^o$ is left artinian. Furthermore, the ring $S^o$ is $G$-graded by $(S^o)_g = S_{g^{-1}}$ (cf. \cite[Rmk. 1.2.4]{nastasescu2004methods}). With this grading, $|\text{Supp}(S^o) | = | \text{Supp}(S) |$. 
But the left case of Theorem \ref{thm:big} implies $\text{Supp}(S^o) < \infty$ and thus $\text{Supp}(S) < \infty$.
\end{remark}

We are now ready to prove our characterization of artinian epsilon-strongly graded rings.

\begin{proof}[Proof of Theorem \ref{thm:a}]
The `if' direction follows from Proposition \ref{prop:suff}. For the other direction, assume $S$ is a left (right) artinian epsilon-strongly graded ring. First note that by Proposition \ref{prop:artinian}, the principal component $R=S_e$ is left (right) artinian. Secondly, by Theorem \ref{thm:big}, $\text{Supp}(S) < \infty$, which is equivalent to $S_g = \{ 0 \}$ for all but finitely many $g \in G$. 
\end{proof}

\begin{remark}
Passman's example \cite{passman1970radicals} of an artinian twisted group rings by an infinite $p$-group shows that Theorem \ref{thm:a} and Theorem \ref{thm:big} do not hold for arbitrary groups.
\end{remark}

\section{Applications to Leavitt path algebras}
\label{sec:4}

A \emph{directed graph} $E=(E^0, E^1, s, r)$ consists of a set of vertices $E^0$, a set of edges $E^1$ and maps $s \colon E^1 \to E^0$, $r \colon E^1 \to E^0$ specifying the source $s(f)$ and range $r(f)$ vertex for each edge $f \in E^1$. We call a graph $E$ \emph{finite} if $E^0$ and $E^1$ are finite sets. A \emph{path} is a sequence of edges $\alpha=f_1 f_2 \dots f_n$ such that $s(f_{i+1})=r(f_i)$ for $1 \leq i \leq n-1$. We write $s(\alpha)=s(f_1)$ and $r(\alpha)=r(f_n)$. A \emph{cycle} is a path such that $s(f_1)=r(f_n)$ and $s(f_i) \neq s(f_1)$ for $2 \leq i \leq n$. 

\begin{definition}
For a directed graph $E= (E^0, E^1, s, r)$ and a ring $R$, the \textit{Leavitt path algebra with coefficients in $R$} is the $R$-algebra generated by the symbols $\{ v \mid v \in E^0 \}$, $\{ f \mid f \in E^1 \}$ and $\{ f^* \mid f \in E^1 \}$ subject to the following relations,
\begin{enumerate}[(a)]
\begin{item}
$ v_i v_j = \delta_{i,j} v_i$ for all $v_i, v_j \in E^0$,
\end{item}
\begin{item}
$s(f) f = f r(f)=f$ and $r(f)f^* = f^*s(f)=f^*$  for all $f \in E^1$,
\end{item}
\begin{item}
$f^* f' = \delta_{f, f'} r(f)$ for all $f, f' \in E^1$,
\end{item}
\begin{item}
$\sum_{f \in E^1, s(f)=v}f f^* =v $ for all $v \in E^0$ for which $s^{-1}(v)$ is non-empty and finite. 
\end{item}
\end{enumerate}
We let $R$ commute with the generators. 
\label{def:lpa}
\end{definition}

The symbols $f \in E^1$ are called \emph{real edges} and the symbols $f^*$ for $f \in E^1$ are called \emph{ghost edges}. For a real path $\alpha= f_1 f_2 f_3 \dots f_n$ the corresponding \emph{ghost path} is $\alpha^* = f_n^* f_{n-1}^* \dots f_3^* f_2^* f_1^*.$ In particular, $s(\alpha)=r(\alpha^*)$ and $r(\alpha)=s(\alpha^*)$. 
Using the relations above it can be proved that a general element in $L_R(E)$ has the form of a finite sum $\sum r_i \alpha_i \beta_i^*$ where $r_i \in R$, $\alpha_i$ and $\beta_i$ are real paths such that $r(\alpha_i)=s(\beta_i^*)=r(\beta_i)$ for every $i$. For a path $\alpha= f_1 f_2 f_3 \dots f_n$, let $\text{len}(\alpha)=n$ denote the length of $n$.  

For any group $G$ there is a class of $G$-gradings called the standard gradings of $L_R(E)$ (see for example \cite{nystedt2017epsilon}). Taking $G=\mathbb{Z}$ we obtain the so-called canonical $\mathbb{Z}$-grading of $L_R(E)$. In this case, the homogeneous components are given by, 
\begin{equation}
 (L_R(E))_n = \Big \{ \sum r_i \alpha_i \beta_i^* \mid r_i \in R, \text{len}(\alpha_i)-\text{len}(\beta_i) = n \Big \},
\label{eq:2}
\end{equation}
for $n \in \mathbb{Z}$. 
It is proved in \cite[Thm. 3.3]{nystedt2017epsilon} that every standard $G$-grading of $L_R(E)$ is symmetric and furthermore when $E$ is finite, every standard $G$-grading is epsilon-strong. In the sequel, we will consider the Leavitt path algebra $L_R(E)$ equipped with the canonical $\mathbb{Z}$-grading. If $E$ is finite, then $L_R(E)$ is an epsilon-strongly $\mathbb{Z}$-graded ring. 

\bigskip

In this section, we will characterize when $L_R(E)$ is noetherian and artinian under the assumption that $E$ is a finite graph.

\begin{definition}
A graph $E$ is said to satisfy \emph{Condition (NE)} if there is no cycle with an exit. 
\end{definition}

This condition is in fact a large restriction on the paths in the graph.

\begin{lemma}(\cite[Lem. 1.2]{abrams2008locally})
If $E$ is a finite graph satisfying Condition (NE), then any path $\mu$ with $| \mu | > | E^0 |$ ends in a cycle. 
\label{lem:fin_ends}
\end{lemma}

We will begin by showing that Condition (NE) allows us to lift algebraic properties of the base ring $R$ to the principal component $(L_R(E))_0$. 

\smallskip

Following \cite[Thm. 1.4]{abrams2008locally}, for an integer $m \geq 0$, let,
\begin{equation*}
C_m = \text{Span}_R \{ \alpha \beta^* \mid \text{len}(\alpha)=\text{len}(\beta)=m \}.
\end{equation*}
Note that by (\ref{eq:2}), $(L_R(E))_0 = \bigcup_{m=0}^\infty C_m$ for any graph $E$. If $E$ is a finite graph, then there are only finitely many paths of a fixed length. Hence, if $E$ is a finite graph, then $C_m$ is a finitely generated $R$-module for each non-negative integer $m$.

\begin{lemma}
(cf. \cite[Thm. 1.4]{abrams2008locally}) If for some integer $t \ge 0$, $C_{t+1} \subseteq \bigcup_{m=0}^t C_m$, then $\bigcup_{m=0}^\infty C_m = \bigcup_{m=0}^t C_m.$
\label{lem:c_inc}
\end{lemma}
\begin{proof}
Assume that $C_{t+1} \subseteq \bigcup_{m=0}^t C_m$. We show that $C_{t+i} \subseteq \bigcup_{m=0}^t C_m$ for any $i$ by induction.

Now, assume that $C_{t+k} \subseteq \bigcup_{m=0}^t C_m$ for some $k$ and let $\alpha = e_1 e_2 \dots e_{t+k+1} f_{t+k+1}^* f_{k}^*\dots f_2^* f_1^*$. Letting $\beta = e_2 \dots e_{t+k+1} f_{t+k+1}^* \dots f_2^*$ we have that,
\begin{equation*}
\alpha = e_1 \beta f_1^* \in e_1 C_{t+k} f_1^* \subseteq e_1 \Big ( \bigcup_{m=0}^t C_m \Big ) f_1^* \subseteq \bigcup_{m=1}^{t+1} C_m \subseteq \bigcup_{m=0}^t C_m,
\end{equation*} where the last inclusion is a consequence of the assumption $C_{t+1} \subseteq \bigcup_{m=0}^t C_m$.
\end{proof}

\begin{proposition}
If $E$ is a finite graph satisfying Condition (NE), then there exists some non-negative integer $k$ such that
$(L_R(E))_0= \bigcup_{m=0}^{k} C_m$. In particular, $(L_R(E))_0$ is a finitely generated $R$-module.
\label{prop:fin_gen}
\end{proposition}
\begin{proof}
Assume that $E$ satisfies condition (NE). We will show that the condition in Lemma \ref{lem:c_inc} holds. 
By Lemma \ref{lem:fin_ends} there is an integer $k > 1$ such that any path $\mu$ in $E$ with $\text{len}(\mu) \geq k$ ends in a cycle. Let $\alpha \beta^* \ne 0 \in C_k$. Then since $r(\alpha)=r(\beta)$ and condition (NE) holds, $\alpha$ and $\beta$ both ends in the same cycle. Let, $\alpha = \alpha' e_1 e_2 \dots e_n$ and $\beta = \beta' e_1 e_2 \dots e_n$ for some edges $e_1, \dots, e_n$. Furthermore, since the cycle does not have an exit, it follows from Definition 4.1(d) that $e_n e_n^* = s(e_n)$. Hence, 
\begin{equation*}
\alpha \beta^* = \alpha' e_1 e_2 \dots e_{n-1} (e_n e_n^*) e_{n-1}^* \dots e_1^* (\beta')^* = \alpha' e_1 \dots e_{n-1} e_{n-1}^* \dots e_1^* (\beta')^* \in C_{k-1}.
\end{equation*} This proves that $C_{k} \subseteq C_{k-1}$. 

Now, Lemma \ref{lem:c_inc} yields, $(L_R(E))_0 = \bigcup_{m=0}^{k-1} C_m$. In particular,  $(L_R(E))_0$ is finitely generated since each $C_m$ is finitely generated.
\end{proof}

As a corollary we obtain that noetherianity lifts from the base ring $R$ to $(L_R(E))_0$. 

\begin{corollary}
Let $E$ be a finite graph and let $R$ be a left (right) noetherian ring. If $E$ satisfies Condition (NE), then $(L_R(E))_0$ is left (right) noetherian as a ring.
\label{cor:noeth_ring}
\end{corollary}
\begin{proof}
Assume that $E$ satisfies Condition (NE). By Proposition \ref{prop:fin_gen}, $(L_R(E))_0$ is finitely generated. Since $R$ is assumed to be left (right) noetherian, $_R (L_R(E))_0, (((L_R(E))_0)_R)$ is noetherian and, in particular $(L_R(E))_0$ is a left (right) noetherian ring. 
\end{proof}

Next, we will show that Condition (NE) in some sense restricts the number of idempotents in $(L_R(E))_0$. In particular, we show that Condition (NE) is a necessary condition for $(L_R(E))_0$ to be one-sided noetherian.

\begin{proposition}
If $E$ does not satisfy Condition (NE), then there exists an infinite sequence of distinct, pairwise orthogonal idempotents $\{ \mu_n \mu_n^* \}_{n \in \mathbb{N}} $  in $(L_R(E))_0$.
If this is the case, then $(L_R(E))_0$ is not left or right noetherian.
\label{prop:idem_exist}
\end{proposition}
\begin{proof}
Following \cite[Prop. 3.1(ii)]{ArandaPino2013}, assume that there is a cycle $\gamma$ is a with an exit path $\alpha$. By shifting the cycle, we may assume that the base vertex of the cycle is the starting point of the exit. More precisely, we may assume that $\gamma = e_1 e_2 \dots e_n$ for edges $e_i$ such that $s(\gamma)=s(e_1)=s(\alpha)$ and $e_i \ne e_1$ for $2 \leq i \leq n$. Consider the set $\{ \gamma^n \alpha \alpha^* (\gamma^*)^n  \mid n \in \mathbb{N} \}$. It is straightforward to prove these are orthogonal idempotents in $(L_R(E))_0$. To see that they are distinct, suppose to get a contradiction that, 
\begin{equation*}
\gamma^n \alpha \alpha^* (\gamma^*)^n = \gamma^m \alpha \alpha^* (\gamma^*)^m,
\end{equation*} for $n > m$. Then, multiplying with $\alpha^* (\gamma^*)^m$ on the left, 
\begin{equation*}
0 = (\alpha^* \gamma^{n-m}) \alpha \alpha^* (\gamma^*)^n = \alpha^* (\gamma^*)^m \gamma^n \alpha \alpha^* (\gamma^*)^n = \alpha^* (\gamma^*)^m \gamma^m \alpha \alpha^* (\gamma^*)^m = \alpha^* (\gamma^*)^m,
\end{equation*} which is a contradiction. 

The last statement is clear since a one-sided noetherian ring does not contain infinitely many orthogonal idempotents.
\end{proof}

\smallskip

We can now prove our main characterization of noetherian Leavitt path algebras, generalizing \cite[Thm. 3.10]{abrams2008locally}.

\begin{corollary}
Let $E$ be a finite graph and let $R$ be a left (right) noetherian ring. Let the Leavitt path algebra $L_R(E)$ be graded with the canonical $\mathbb{Z}$-grading. Then the following are equivalent:
\begin{enumerate}[(a)]
\begin{item}
$E$ satisfies Condition (NE),
\end{item}
\begin{item}
$(L_R(E))_0$ is left (right) noetherian,
\end{item}
\begin{item}
$L_R(E)$ is left (right) noetherian.
\end{item}
\end{enumerate}
\label{cor:lpa_noeth}
\end{corollary}
\begin{proof}
$(a) \implies (b):$  Assume that $E$ has no cycle with an exit. Since $R$ is left (right) noetherian, by Corollary \ref{cor:noeth_ring}, $(L_R(E))_0$ is left (right) noetherian.

$(b) \implies (c):$ Since the canonical $\mathbb{Z}$-grading of $L_R(E)$ is epsilon-strong (see \cite[Thm. 3.3]{nystedt2017epsilon}), this follows from Theorem \ref{thm:n}.

$(c) \implies (a):$ Assume that $L_R(E)$ is left (right) noetherian. Then by Theorem \ref{thm:n}, $(L_R(E))_0$ is left (right) noetherian. Therefore, by Proposition \ref{prop:idem_exist}, $E$ satisfies Condition (NE).
\end{proof}

\begin{remark}
Note that if $R$ is two-sided noetherian in Corollary \ref{cor:lpa_noeth}, then we get equivalences between the following assertions:
\begin{enumerate}[(a)]
\begin{item}
$E$ satisfies Condition (NE),
\end{item}
\begin{item}
$(L_R(E))_0$ is left noetherian,
\end{item}
\begin{item}
$(L_R(E))_0$ is right noetherian,
\end{item}
\begin{item}
$L_R(E)$ is left noetherian,
\end{item}
\begin{item}
$L_R(E)$ is right noetherian.
\end{item}
\end{enumerate}
\end{remark}
\begin{remark}
Taking $R$ to be a field we obtain a new proof of Abrams, Aranda Pino and Siles Molina's characterization of noetherian Leavitt path algebras of finite graphs with coefficients in a field (cf. \cite{abrams2008locally}).
\end{remark}

\bigskip

Next, we will similarly apply the characterization of artinian epsilon-strongly graded rings. We first show that Condition (NE) will allow us to lift additional algebraic properties from $R$ to $(L_R(E))_0$. This follows as a corollary to a more general structure theorem for the principal component $(L_R(E))_0$. Define a filtration of $(L_R(E))_0$ as follows. For $n \geq 0 $, put,
\begin{equation*}
D_n = \text{Span}_R \{ \alpha \beta^* \mid \text{len}(\alpha) = \text{len}(\beta) \leq n \}.
\end{equation*} It is straightforward to show that $D_n$ is an $R$-subalgebra of $(L_R(E))_0$. For $v \in E^0$ and $n > 0$ let $P(n,v)$ denote the set of paths $\gamma$ with $\text{len}(\gamma)=n$ and $r(\gamma)=v$. Let $\text{Sink}(E)$ denote the set of sinks in $E$. 

\begin{theorem}(\cite[Cor. 2.1.16]{abrams2017leavitt})
For a non-negative integer $n$, let $M_n(R)$ denote the full $n \times n$-matrix ring. Then, for a finite graph $E$, 
\begin{equation*}
D_0  \simeq \prod_{v \in E^0} R ,
\label{eq:u5}
\end{equation*} 
\begin{equation*}
D_n \simeq \prod_{\substack{0 \leq i \leq n-1 \\ v \in \text{Sink}(E)}} M_{| P(i, v) |}(R) \times \prod_{v \in E^0} M_{ | P(n, v) | }(R),
\label{eq:u6}
\end{equation*}
as $R$-algebras.
\label{thm:dn}
\end{theorem}

We now show that we can lift even more properties from the base ring $R$ to $(L_R(E))_0$ when Condition (NE) is satisfied.

\begin{corollary}
Let $R$ be a ring and let $E$ be a finite graph that satisfies Condition (NE). Then $(L_R(E))_0$ is Morita equivalent to $R^k$ for some $k > 0$ and, in particular, the following assertions hold:
\begin{enumerate}[(a)]
\begin{item}
$R$ is semisimple if and only if $(L_R(E))_0$ is semisimple.
\end{item}
\begin{item}
$R$ is von Neumann regular if and only if $(L_R(E))_0$ is von Neumann regular.
\end{item}
\begin{item}
$R$ is left (right) noetherian if and only if $(L_R(E))_0$ is left (right) noetherian.
\end{item}
\begin{item}
$R$ is left (right) artinian if and only if $(L_R(E))_0$ is left (right) artinian.
\end{item}
\end{enumerate} 
\label{cor:ne_art}
\end{corollary}
\begin{proof}
Assuming that $E$ is finite and satisfies Condition (NE), Proposition \ref{prop:fin_gen} implies that $(L_R(E))_0=\bigcup_{m=0}^k C_m = D_k$, which is a finite direct product of full matrix rings by Theorem \ref{thm:dn}. That is, $(L_R(E))_0 = M_{n_1}(R) \times M_{n_2}(R) \times \dots \times M_{n_k}(R)$ for some positive integers $n_1, n_2, \dots, n_k$. 

Recall that $R$ and the full matrix ring $M_m(R)$ are Morita equivalent for each $m$ (see for example \cite[18.6]{lam2012lectures}), i.e. there is an equivalence of categories $M_m(R)\textrm{-}\textsf{Mod} \approx R\textrm{-}\textsf{Mod}$. In particular, we have equivalences $M_{n_i}(R)\textrm{-}\textsf{Mod} \approx R\textrm{-}\textsf{Mod}$ for each $1 \leq i \leq k$. Forming the product categories, it is not hard to see that, 
\begin{equation*}
M_{n_1}(R)\textrm{-}\textsf{Mod} \times M_{n_2}(R)\textrm{-}\textsf{Mod} \times \dots \times M_{n_k}(R)\textrm{-}\textsf{Mod} \approx  R\textrm{-}\textsf{Mod} \times R\textrm{-}\textsf{Mod} \times \dots \times R\textrm{-}\textsf{Mod}.
\end{equation*}
For any rings $R,S$ it can be shown that $R\textrm{-}\textsf{Mod} \times S\textrm{-}\textsf{Mod} \approx (R \times S) \textrm{-}\textsf{Mod}$. Hence, 
\begin{equation*}
(M_{n_1}(R) \times \dots \times M_{n_k}(R))\textrm{-}\textsf{Mod} \approx R^k \textrm{-}\textsf{Mod},
\end{equation*}
and thus $(L_R(E))_0$ and $R^k$ are Morita equivalent. 

Furthermore, it is well-known that $R^k$ is semisimple/von Neumann regular/one-sided artinian/one-sided noetherian if and only if $R$ is semisimple/von Neumann regular/one-sided artinian/one-sided noetherian. Moreover, these properties are Morita invariant, which is enough to establish the corollary. 
\end{proof}

\smallskip

We will begin by showing a partial result concerning the characterization of artinian Leavitt path algebras. 

\begin{lemma}
Let $E$ be a finite graph. Then $E$ is acyclic if $E$ satisfies Condition (NE) and contains no infinite paths.
\label{lem:acyclic_noexit}
\end{lemma}
\begin{proof}
Assume that $E$ satisfies Condition (NE). Then any infinite path must end in a cycle, hence $E$ contains infinite paths if and only if $E$ contains a cycle. 
\end{proof}

The following is our first result concerning artinian Leavitt path algebras. 

\begin{proposition}
Let $E$ be a finite graph and let $R$ be a left (right) artinian ring. Let $L_R(E)$ be the Leavitt path algebra graded by the canonical $\mathbb{Z}$-grading. Then the following assertions are equivalent:
\begin{enumerate}[(a)]
\begin{item}
$E$ is acyclic, 
\end{item} 
\begin{item}
$(L_R(E))_0$ is left (right) artinian and $\text{Supp}(L_R(E))$ is finite,
\end{item}
\begin{item}
$L_R(E)$ is left (right) artinian.
\end{item}
\end{enumerate}
\label{prop:lpa_left}
\end{proposition}
\begin{proof}
$(a) \implies (b):$ Assume that $E$ is acyclic. Then, in particular, $E$ trivially satisfies Condition (NE), so by Corollary \ref{cor:ne_art}, $(L_R(E))_0$ is left (right) artinian. 
We claim that there exists some $n$ satisfying $(L_R(E))_k = \{ 0 \}$ for all $k \in \mathbb{Z}$ such that $|k| > n$. 
Since $E$ is finite and acyclic there is a maximal length of paths in $E$. By (\ref{eq:2}), a monomial $\alpha \beta^* \in (L_R(E))_k$ if and only if $\text{len}(\alpha) - \text{len}(\beta)=k$. Taking $n$ larger than the maximal path length in $E$, it follows that $(L_R(E))_k = \{ 0 \}$ for all $k$ such that $|k| > n$, thus proving the claim.

$(b) \implies (c):$ Since the canonical $\mathbb{Z}$-grading on $L_R(E)$ is epsilon-strong (see \cite[Thm. 3.3]{nystedt2017epsilon}), Theorem \ref{thm:a} implies that $L_R(E)$ is left (right) artinian.

$(c) \implies (a):$ Assume that $L_R(E)$ is left (right) artinian. Then Theorem \ref{thm:a} implies that (i) $(L_R(E))_0$ is left (right) artinian and (ii) there is $n$ such that $(L_R(E))_k= \{ 0 \}$ for all $k \in \mathbb{Z}$ such that $k > n$. In particular, (i) implies that $(L_R(E))_0$ is left (right) noetherian, so $E$ satisfies Condition (NE) by Corollary \ref{cor:lpa_noeth}. We claim that (ii) implies that $E$ does not contain any infinite paths. Assuming that the claim hold, Lemma \ref{lem:acyclic_noexit} proves that $E$ is acyclic. 
Suppose to get a contradiction that $E$ contains an infinite path. Then, we can construct finite paths of arbitrary length by taking initial subpaths of the infinite path. Hence, $(L_R(E))_k \ne \{ 0 \}$ for arbitrarily large $k$, contradicting (ii). 
\end{proof}
\begin{remark}
By making the stronger assumption that $R$ is semisimple we see by Corollary \ref{cor:ne_art} that condition $(b)$ can be replaced by,
\begin{enumerate}[(b')]
\begin{item}
$(L_R(E))_0$ is semisimple and $\text{Supp}(L_R(E))$ finite.
\end{item}
\end{enumerate}
\end{remark}

To get our full characterization we will apply a theorem by Nystedt, Öinert and Pinedo, which will allow us to lift semisimplicity from the principal component to the whole ring. Let $S=\bigoplus_{i \in \mathbb{Z}} S_i$ be an arbitrary epsilon-strongly $\mathbb{Z}$-graded ring and let $Z(S_0)_{\text{fin}}$ denote the elements $r \in Z(S_0)$ such that $r \epsilon_i = 0$ for all but finitely many integers $i$. There is a way to define a trace function $\text{tr}_\gamma \colon Z(S_0)_{\text{fin}} \to Z(S_0)$  (see \cite[Def. 14]{nystedt2016epsilon}) such that the following theorem holds.

\begin{theorem}(\cite[Thm. 23]{nystedt2016epsilon})
Let $S$ be an epsilon-strongly $\mathbb{Z}$-graded ring. Assume that $\epsilon_i = 0$ for all but finitely many integers $i$ and that $\text{tr}_\gamma(1)$ is invertible in $S_0$. If $S_0$ is semisimple, then $S$ is semisimple.
\label{thm:semisimple}  
\end{theorem}
\begin{remark}
Note that Nystedt, Öinert and Pinedo show Theorem \ref{thm:semisimple} for an epsilon-strongly $G$-graded ring, where $G$ is an arbitrary group.
\end{remark}
The theorem only gives sufficient conditions: $\text{tr}_\gamma(1)$ is not necessarily invertible in $S_0$ if $S$ is semisimple. We will need the following to be able to apply the theorem.

\begin{lemma}
Let $E$ be a finite graph and $R$ be a ring such that $n \cdot 1_R$ is invertible for every integer $n \ne 0$. If $\text{Supp}(L_R(E))$ is finite, then $\text{tr}_\gamma(\sum_{ v \in E^0 } v)$ is invertible in $(L_R(E))_0$.\label{lem:invertible}
\end{lemma}
\begin{proof}
Assume that $\text{Supp}(L_R(E))$ is finite.  
Note that the multiplicative identity of $(L_R(E))_0$ is $\epsilon_0 = \sum_{v \in E^0} v$ where the sum is well-defined since $E$ is a finite graph. By the definition of the trace (\cite[Def. 14]{nystedt2016epsilon}), $\text{tr}_\gamma(\epsilon_0) = \sum_{i \in \mathbb{Z}} \epsilon_i$. Indeed, this expression is valid since $\epsilon_i = 0$ for all but finitely many integers $i$ by the assumption that the support is finite.

By \cite[Thm. 3.3]{nystedt2017epsilon}, every non-zero $\epsilon_i$ can be written as a finite sum,
\begin{equation*}
\epsilon_i = \sum_i v_i + \sum_i \alpha_i \alpha_i^*,
\end{equation*} 
where $\alpha_i$ are paths. Let $v_1, \dots, v_{| E^0 |}$ be the vertices of $E$. Since $\epsilon_0$ is a term in the trace, there are two sequences of positive integer $(n_i)_{i=1}^{| E^0 |}, (n_i')_{i=1}^{c'}$ and a sequence of paths $(\alpha_i)_{i=1}^{c'}$ such that,
\begin{equation*}
\text{tr}_\gamma(\epsilon_0) = \sum_{i=1}^{| E^0 |} n_i v_i + \sum_{i=1}^{c'} n_i' \alpha_i \alpha_i^*.
\end{equation*}
We will show that the right-hand side is invertible in $(L_R(E))_0$. Recall that $v_i v_j = \delta_{i,j}v_i$ and $v_i \alpha_j = \delta_{v_i, s(\alpha_j)} \alpha_j$ for all indices $i,j$. For $1 \leq i,j \leq c'$, write $\alpha_i \leq \alpha_j$ if $\alpha_i$ is an initial subpath of $\alpha_j$. A moment's thought yields that $(\alpha_i \alpha_i^*)(\alpha_j \alpha_j^*)=(\alpha_j \alpha_j^*)(\alpha_i \alpha_i^*) = \alpha_j \alpha_j^*$ if $\alpha_i \leq \alpha_j$ and $0$ otherwise. Let $T$ be the set of triples $(i,j,t)$ of indices such that $\alpha_i \leq \alpha_j$ or $\alpha_j \leq \alpha_i$ and let $t$ be the index of the longer path. Finally, let $f \colon \{ 1, 2, \dots, c' \} \to \{ 1, 2, \dots, |E^0| \}$ be the function satisfying $s(\alpha_i)=v_{f(i)}$ for all $1 \leq i \leq c'$. In other words, $f$ maps the path index $i$ to the index of the start vertex $s(\alpha_i)$. 

Now, making an Ansatz for a right inverse and calculating gives,
\begin{align*}
\Big ( \sum_{i=1}^{| E^0 |} n_i v_i + \sum_{i=1}^{c'} n_i' \alpha_i \alpha_i^* \Big ) \Big ( \sum_{i=1}^{| E^0 |} m_i v_i + \sum_i^{c'} m_i' \alpha_i \alpha_i^* \Big ) \\ 
= \sum_{i=1}^{| E^0 |} n_i m_i v_i + \sum_{i=1}^{c'} (n_{f(i)} m_i' + n_i' m_{f(i)}) \alpha_i \alpha_i^* + \sum_{ (i,j,t) \in T } n_i' m_j' \alpha_t \alpha_t^*. 
\end{align*}
Letting $m_i = 1/n_i$ for all $i$ we see that the first sum equals $\epsilon_0 = \sum_{i=1}^{| E^0 |} v_i$. Next, we will solve for $m_i'$ with the aim to kill the second and third sums. Consider the coefficient of $\alpha_{t} \alpha_{t}^*$ for some index $t$. We require that,
\begin{align*}
0 &= n_{f(t)}m_{t}' + n_{t}' m_{f(t)} + \sum_{\substack{ (i,j,t) \in T }} n_i' m_j'   \\
 &= n_{f(t)}m_{t}' + n_{t}' m_{f(t)} + \Big ( \sum_{ \substack{ (t,j,t) \in T \\ j \ne t }} n_i' m_j' + \sum_{  \substack{ (i,t,t) \in T } } n_i' m_{t}' \Big ) \\
 &= \Big ( n_{f(t)} + \sum_{ \substack{ (i,t,t) \in T }} n_i' \Big ) m_{t}' + n_{t}' m_{f(t)} + \sum_{ \substack{ (t,j,t) \in T \\ j \ne t }} n_i' m_j'.
\end{align*}

Rearranging, this means we need to solve for $(m_i')_{i=1}^{c'}$ in the following set of simultaneous equations,
\begin{equation}
m_{t}' = - \Big ( n_{f(t)} + \sum_{(i,t,t) \in T} n_i' \Big )^{-1} \Big ( \frac{n_{t}'}{n_{f(t)}} + \sum_{ \substack{(t,j,t) \in T \\ j \ne t }} n_{t}' m_j' \Big ),
\label{eq:bm4}
\end{equation}
where $1 \leq t \leq c'$. 

We will find a solution to this system of equations by solving a set of independent subsystems. Consider the chain decomposition of the poset $A = \{ \alpha_1, \alpha_2, \dots, \alpha_{c'} \}$ equipped with the ordering defined above. More precisely, let $S$ be the minimal elements of $A$ and for each $\alpha_k \in S$, let $P_k : \alpha_k \leq \alpha_{k_1} \leq \dots \leq \alpha_{k_n}$ be the longest chain starting at $\alpha_k$. Note that the family of chains $(P_k)_{\alpha_k \in S}$ is a partition of $A$.

Now, consider a fixed chain $P_{k_0} : \alpha_{k_0} \leq \alpha_{k_1} \leq \dots \leq \alpha_{k_n}$ and put $I = \{ k_0, k_1, \dots, k_n \}$. We claim that the set of equations (\ref{eq:bm4}) where $t \in I$ is an independent subsystem. Indeed, a moment's thought yields that this is a system in the variables $\{ m_t' \mid t \in I \}$ and conversely these variables only appear in the equations for which $t \in I$. More precisely, since $\alpha_{k_0}$ is minimal, there is no triple $(k_0,j,k_0) \in T$ such that $j \ne k_0$. Hence, taking $t=k_0$ in (\ref{eq:bm4}), the second sum is empty and we can solve directly for $m_{k_0}'$. Moreover, for $0 \leq i \leq n$ note that $(k_i, j, k_i) \in T$ if and only if $j \in \{ k_0, k_1, \dots, k_i \}$. Hence, taking $t=k_i$ in (\ref{eq:bm4}), we can solve for $m_{k_i}'$ in terms of $m_{k_0}', m_{k_1}', \dots, m_{k_{i-1}}'.$ Going along the chain and successively solving the equations, we obtain a solution to the subsystem (\ref{eq:bm4}) for the indices in $P_{k_0}$. Repeating this process for each chain, we obtain a solution to the whole system. 

Thus, we obtain a right inverse of $\text{tr}_\gamma(\epsilon_0)$. Since $\text{tr}_\gamma(\epsilon_0) \in Z((L_R(E))_0)$, this is enough to establish the lemma.

\end{proof}

%

The following is one of our main result:

\begin{corollary}
Let $E$ be a finite graph and let $R$ be a semisimple ring such that $n \cdot 1_R$ is invertible for every integer $n \ne 0$. Let $L_R(E)$ be the corresponding Leavitt path algebra graded by the canonical $\mathbb{Z}$-grading. Then the following assertions are equivalent:
\begin{enumerate}[(a)]
\begin{item}
$E$ is acyclic,
\end{item}
\begin{item}
$(L_R(E))_0$ is semisimple and $\text{Supp}(L_R(E))$ is finite,
\end{item}
\begin{item}
$L_R(E)$ is left artinian, 
\end{item}
\begin{item}
$L_R(E)$ is right artinian, 
\end{item}
\begin{item}
$L_R(E)$ is semisimple, 
\end{item}
\end{enumerate}
\label{cor:full_semi}
\end{corollary}
\begin{proof}
$(a) \iff (b) \iff (c):$ Proposition \ref{prop:lpa_left}.
 
$(b) \implies (e):$ Assume that $(L_R(E))_0$ is semisimple and $\text{Supp}(L_R(E))$ is finite. By Lemma \ref{lem:invertible} and Theorem \ref{thm:semisimple}, $L_R(E)$ is semisimple.

$(d) \implies (c):$ Assume that $L_R(E)$ is right artinian. By the Hopkins–Levitzki theorem, $L_R(E)$ is left artinian if and only if $L_R(E)$ is left noetherian. Since $R$ is semisimple, it is two-sided noetherian. Hence, since $L_R(E)$ is right noetherian, Corollary \ref{cor:lpa_noeth} implies $L_R(E)$ is also left noetherian. 

$(e) \implies (c), (e) \implies (d):$ Standard consequences of the Artin-Wedderburn theorem. 
\end{proof}

\begin{remark}
As a special case of Corollary \ref{cor:full_semi} we obtain a characterization of Leavitt path algebras over a field of characteristic $0$. This is consistent with Abrams, Aranda Pino and Siles Molina's characterization of artinian and semisimple Leavitt path algebras with coefficients in a field, cf. \cite{ABRAMS2007753,abrams2010chain,abrams2017leavitt}. However, their characterization does not require that the base field has characteristic $0$. In particular, this means that the assumption on the base ring $R$ in Corollary \ref{cor:full_semi} is not a necessary condition. 
\end{remark}

\smallskip

Finally, we will prove Theorem \ref{thm:s}, which is a generalization of Steinberg's characterization of Leavitt path algebras (cf. \cite{steinberg2018chain}). We shall first show that noetherian unital Leavitt path algebras come from finite graphs. 

\begin{lemma}
If $(L_R(E))_0$ is left (right) noetherian, then $E$ is finite and satisfies Condition (NE).
\label{lem:e_finite}
\end{lemma}
\begin{proof}
Assume that $(L_R(E))_0$ is left (right) noetherian. If we can prove that $E$ is finite, then the statement follows from Proposition \ref{prop:idem_exist}.

Suppose to get a contradiction that $E$ is not finite. 
If $E^0$ is an infinite set, we can find an infinite sequence of pairwise orthogonal idempotents $\{ v_i \}_{i \in \mathbb{N}}$ in $(L_R(E))_0$ contradicting that $(L_R(E))_0$ is left (right) noetherian. Similarly, if $E^1$ is infinite, it is straightforward to check that $\{ e e^* \}_{e \in E^1}$ is an infinite sequence of orthogonal idempotents. Hence, $E$ is finite.
\end{proof}

Before finishing the proof we recall the following well-known property of group graded rings. Let $G$ be an arbitrary group and let $S=\bigoplus_{g \in G} S_g$ be a $G$-graded ring with principal component $R$. If $S$ is semisimple, then $R$ is semisimple.

\begin{proof}[Proof of Theorem \ref{thm:s}]
$(a):$ Assume that $L_R(E)$ is left (right) noetherian. By Theorem \ref{thm:n}, $(L_R(E))_0$ is left (right) noetherian and hence by Lemma \ref{lem:e_finite}, $E$ is finite and satisfies Condition (NE). Furthermore, by Corollary \ref{cor:ne_art}, $R$ is left (right) noetherian. The converse follows directly from Corollary \ref{cor:lpa_noeth}.

$(b):$ Assume that $L_R(E)$ is left (right) artinian. By Theorem \ref{thm:a}, in particular, $(L_R(E))_0$ is left (right) artinian and therefore also left (right) noetherian. By Lemma \ref{lem:e_finite}, $E$ is finite and satisfies Condition (NE). This means we can apply Corollary \ref{cor:ne_art} to get that $R$ is left (right) artinian. Hence, Proposition \ref{prop:lpa_left} proves $E$ is acyclic. The converse follows directly from Proposition \ref{prop:lpa_left}.

$(c):$ Assume that $L_R(E)$ is semisimple. This implies that $(L_R(E))_0$ is semisimple. On the other hand, $L_R(E)$ is, in particular, left artinian, hence by $(b)$, $E$ is finite acyclic and hence satisfies Condition (NE). Thus, Corollary \ref{cor:ne_art} implies that $R$ is semisimple. The converse follows directly from Corollary \ref{cor:full_semi}. 
\end{proof}

\section{Applications to unital partial crossed products}
\label{sec:5}

Let $R$ be an associative, non-trivial unital ring and let $G$ be a group with neutral element $e$. A \emph{unital twisted partial action of $G$ on $R$} (see \cite[pg. 2]{nystedt2016epsilon}) is a triple, 
\begin{equation*}
( \{\alpha_g \}_{g \in G}, \{ D_g \}_{g \in G}, \{ w_{g,h} \}_{(g,h) \in G \times G}),
\end{equation*}
where for each $g \in G$, the $D_g$'s are unital ideals of $R$, $\alpha_g \colon D_{g^{-1}} \to D_g$ are ring isomorphisms and for each $(g, h) \in G \times G$, $w_{g,h}$ is an invertible element in $D_g D_{g h}$. Let $1_g \in Z(R)$ denote the (not necessarily non-zero) multiplicative identity of the ideal $D_g$. We require that the following conditions hold for all $g, h \in G$:
\begin{itemize}
\item[(P1)] $\alpha_e = {\rm id}_R$;
\item[(P2)] $\alpha_g(D_{g^{-1}} D_h) = D_g D_{gh}$;
\item[(P3)] if $r \in D_{h^{-1}} D_{(gh)^{-1}}$, 
then $\alpha_g ( \alpha_h (r) ) =
w_{g,h} \alpha_{gh}(r) w_{g,h}^{-1}$;
\item[(P4)] $w_{e,g} = w_{g,e} = 1_g$;
\item[(P5)] if $r \in D_{g^{-1}} D_h D_{hl}$, then
$\alpha_g(r w_{h,l}) w_{g,hl} =
\alpha_g(r) w_{g,h} w_{gh,l}$.
\end{itemize}

Given a unital twisted partial action of $G$ on $R$, we can form the \emph{unital partial crossed product} $R \star_\alpha^w G = \bigoplus_{g \in G} D_g \delta_g$ where the $\delta_g$'s are formal symbols. For $g,h \in G, r \in D_g$ and $r' \in D_h$ the multiplication is defined by the rule:
\begin{itemize}
\item[(P6)] $(r \delta_g) (r' \delta_h) = r \alpha_g(r' 1_{g^{-1}}) w_{g,h} \delta_{gh}$.
\end{itemize}
Directly from the definition of the multiplication it follows that $1_R \delta_e$ is the multiplicative identity of $R \star_\alpha^w G$. It can also be proved that $R \star_\alpha^w G$ is an associative $R$-algebra (see \cite[Thm. 2.4]{dokuchaev2008crossed}). 
Moreover, as mentioned in the introduction, Nystedt, Öinert and Pinedo \cite{nystedt2016epsilon} shows that the natural $G$-grading is epsilon-strong. Furthermore, note that the principal component of $R \star_\alpha^\omega G$ can be identified with $R$.

\smallskip

Our characterization of noetherian epsilon-strongly graded rings (Theorem \ref{thm:n}) gives the following generalization of Theorem \ref{thm:partial_skew_groups_noeth}.

\begin{corollary}
If $G$ is a polycyclic-by-finite group, then $R \star_\alpha^\omega G$ is left (right) noetherian if and only if $R$ is left (right) noetherian.
\label{thm:crossed_noeth}
\end{corollary}
\begin{proof}
Since $R \star_\alpha^\omega G$ is epsilon-strongly $G$-graded (see \cite{nystedt2016epsilon}) by the polycyclic-by-finite group $G$, the statement follows from Theorem \ref{thm:n}.
\end{proof}
\begin{remark}
We show that Theorem \ref{thm:partial_skew_groups_noeth} (cf. \cite[Cor. 3.4]{carvalho2011partial}) follows as a special case of Corollary \ref{thm:crossed_noeth}. Let $\alpha = \{ \alpha_g \colon D_{g^{-1}} \to D_g \}_{g \in G}$ be a partial action of $G$ on $R$ such that each ideal $D_g$ of $R$ is unital. Taking $w_{g,h}=1_g 1_{gh}$ it becomes a unital twisted partial action. The above theorem yields the required statement for the skew group ring $R \star_\alpha G$. 
\end{remark}

\smallskip

Applying our characterization of artinian epsilon-strongly graded rings (Theorem \ref{thm:a}) we obtain the following:

\begin{corollary}
Let $G$ be a torsion-free group. Then $R \star_\alpha^\omega G$ is left (right) artinian if and only if $R$ is left (rigth) artinian and $D_g = \{ 0 \}$ for all but finitely many $g \in G$. 
\label{thm:crossed_artinian}
\end{corollary}
\begin{proof}
Since $R \star_\alpha^\omega G$ is epsilon-strongly $G$-graded (see \cite{nystedt2016epsilon}) by the torsion-free group $G$, the statement follows from Theorem \ref{thm:a}.
\end{proof}

\begin{remark}
Note that Passman's example \cite{passman1970radicals} of an artinian twisted group ring by an infinite $p$-group shows that Corollary \ref{thm:crossed_artinian} does not hold for an arbitrary group $G$. 
\end{remark}

\section*{acknowledgement}
The author is grateful to Johan Öinert, Stefan Wagner and Patrik Nystedt for giving comments and feedback that helped to improve this manuscript.

\bibliographystyle{plain}
\bibliography{chain}

\end{document}